\documentclass[11pt]{amsart}
\usepackage[colorlinks=true, pdfstartview=FitV, linkcolor=blue, citecolor=blue, urlcolor=blue, breaklinks=true]{hyperref}
\usepackage{amsmath,amsfonts,amssymb,amsthm,mathrsfs,amscd,comment,paralist,stmaryrd,etoolbox,mathtools,enumitem}
\usepackage[usenames,dvipsnames]{color}
\usepackage{booktabs}
\usepackage[all]{xy}
\usepackage{tikz}
\usetikzlibrary{decorations.pathmorphing}
\usepackage{capt-of,caption,xcolor,tcolorbox}
\usetikzlibrary{arrows.meta, positioning,fit,backgrounds}
\usepackage{pdflscape} 
\theoremstyle{plain}
\newtheorem{thm}{Theorem}[section]

\newtheorem{lemma}[thm]{Lemma}
\newtheorem{cor}[thm]{Corollary}

\theoremstyle{definition}

\newtheorem{rmk}[thm]{Remark}

\newtheorem{nt}[thm]{Note}
\numberwithin{equation}{section}

\newcommand{\al}{\alpha}
\newcommand{\be}{\beta}
\newcommand{\ga}{\gamma}
\newcommand{\la}{\lambda}
\newcommand{\om}{\omega}
\newcommand{\del}{\delta}
\newcommand{\ep}{\epsilon}

\newcommand{\Z}{\mbox{$\mathbb Z$}}

\newcommand{\C}{\mbox{$\mathbb C$}}     

\newcommand{\g}{\mathfrak{g}}
\newcommand{\h}{\mathfrak{h}}

\newcommand{\gl}{\mathfrak{gl}}

\newtoggle{comments}
\newtoggle{details}
\newtoggle{prelimnote}
\newtoggle{detailsnote}

\toggletrue{detailsnote}   

\iftoggle{comments}{%
	\newcommand{\comments}[1]{
		\begin{center}
			\parbox{6.5 in}{
				\color{red}
				{\footnotesize \textbf{Comments:} #1}
				\color{black}}
	\end{center}}
}{%
	\newcommand{\comments}[1]{}
}

\iftoggle{details}{%
	\newcommand{\details}[1]{
		\ \\
		\color{OliveGreen}
		\begin{footnotesize}
			\textbf{Details:} #1
		\end{footnotesize}
		\color{black}
		\\
	}
}{%
	\newcommand{\details}[1]{}
}

\iftoggle{prelimnote}{%
	
}{%
	
}

\begin{document}
	%
	
	\title{On the Canonical Construction of Simple Lie Superalgebras}
	
	\author{J. Dhamothiran and Saudamini Nayak}
	\address{ Department of Mathematics, National Institute of Technology Calicut, NIT Campus P.O., Kozhikode-673 601, India}
	\email{dhamumjt@gmail.com, dhamothiran\_p230063ma@nitc.ac.in}
	\address{ Department of Mathematics, National Institute of Technology Calicut, NIT Campus P.O., Kozhikode-673 601, India}
	\email{saudamini@nitc.ac.in}

	\thanks{}

	\begin{abstract}
	Axioms for the generalization of root systems were defined and classified (irreducible) by V. Serganova, which precisely correspond to the root systems of basic classical Lie Superalgebras. Here, we present a unified method for constructing simple Lie Superalgebras from the abstract root system, with the choice of base having the minimal number of isotropic roots.	
	\end{abstract}
	
\subjclass[2020]{17B10, 17B22, 17B65.}
\keywords{Generalized root systems, Lie superalgebras}
	
	\maketitle
	\thispagestyle{empty}
	
	\setcounter{tocdepth}{1}
	
	
	%
	\section{Introduction}
	
	%
	Given a complex semisimple Lie algebra $\g$, one can achieve a root space decomposition by fixing a Cartan subalgebra. In this process, to each $\g$ one associates an object called a root system, which is classified by the Cartan matrix and Dynkin diagram attached to it. Thus, complex semisimple Lie algebras are classified by their Dynkin diagrams; these results were achieved between 1890 and 1900 through the work of E. Cartan and W. Killing. Later, root systems are studied extensively and independently due to their rich algebraic and combinatorial structures.
	
Meanwhile, starting with a root system it is natural to ask  how to retrieve the corresponding semisimple Lie algebra. In 1966, J. Tits achieved this by a subtle choice of signs in Chevalley's integral basis, and J. P. Serre defined it by taking a suitable quotient of a free Lie algebra. In 1967, V. G. Kac and R. V. Moody introduced another method for obtaining a simple Lie algebra from a generalized Cartan matrix. The corresponding algebra is called a Kac-Moody algebra and it can be extended to any arbitrary matrix. Correspondingly, one can obtain a contragredient Lie algebra.
	
	Following this, a contragredient Lie superalgebra was introduced as a generalization of a contragredient Lie algebra. In 1977, V. G. Kac classified finite dimensional simple contragredient Lie superalgebras \cite{kac1977lie, kac1977sketch}. In general, these algebras are close to being simple and not necessarily finite dimensional. 
	
	In this article, we present a unified approach to construct simple Lie superalgebras, motivated by the work of M. Geck \cite{geck2017construction}, who achieved a similar result in the context of Lie algebras. In particular, we construct a simple Lie superalgebra as a subalgebra of the general linear Lie superalgebra consisting of all linear maps on a complex vector superspace $M$ having a homogeneous Chevalley-type basis induced by a root system. For this purpose, we begin with the notion of generalized root systems (in short, GRS), as introduced by V. Serganova  \cite{serganova1996generalizations}. In that article, the author classified irreducible generalized root systems with isotropic roots, which coincide with the root systems of the basic classical Lie superalgebras from the Kac's list (with the exception of $A(1, 1)$ and $B(0, n)$), which are of contragredient type.
	
	Following work of Serganova various root systems are introduced. For instance, M. Yousofzadeh introduced extended affine root supersystems, \cite{yousofzadeh2016extended}, locally finite root supersystems, \cite{yousofzadeh2017locally}. Another combinatorial object, generalized reflection root system is introduced by M. Gorelik \cite{gorelik2017generalized}. These root systems are not finite (in general), and an irreducible finite one is exactly an irreducible GRS; this identification provides more flexibility in working with GRSs. 
	
	Section \ref{sec2} contains some preliminary results about root strings in a GRS $R$. In section \ref{sec3} we define specific homogeneous linear maps $e_i,f_i\text{ and }h_i$ on a finite-dimensional vector superspace. These maps satisfy relations that are expected of the Chevalley generators of a simple Lie superalgebra. In section \ref{sec4}, it is shown that $e_i,f_i$ indeed generate a simple Lie superalgebra $\g$ with root system isomorphic to $R$. Finally, in \ref{sec5} we sketch the quiver diagram of the positive system of all irreducible GRS, to emphasize some results.
	
	%
	\section{Preliminaries}\label{sec2}
	%
	
	We begin this section by recalling basic facts about generalized root systems; \cite{serganova1996generalizations}. Let $V$ be a finite dimensional complex vector space with a non-degenerate bilinear inner product $(\cdot,\cdot)$. A finite set $R\subset V\backslash\{0\}$ is called a generalized root system if the following holds.
		\begin{enumerate}
			\item[(1)] $R$ spans $V$ and $R=-R$;
			\item[(2)] if $\alpha,\beta\in R$ and $(\alpha,\alpha)\neq0$ then $(\beta,\alpha^{\vee}):=\dfrac{2(\beta,\alpha)}{(\alpha,\alpha)}\in\Z$ and $r_{\alpha}(\beta)=\beta-(\beta,\alpha^{\vee})\alpha\in R$;
			\item[(3)] if $\alpha\in R$ and $(\alpha,\alpha)=0$ then there exists an invertible map $r_{\alpha}:R\rightarrow R$ such that $r_{\alpha}(\beta)=\beta \text{ if }(\beta,\alpha)=0\text{ and }r_{\alpha}(\beta)\in\{\be+\al,\be-\al\}\text{ otherwise}$.
		\end{enumerate}
Let $R^{im}=\{\alpha\in R:(\alpha,\alpha)=0\} \text{ set of isotropic roots},R^{re}=\{\alpha\in R:(\alpha,\alpha)\neq0\}$, $\alpha^{\vee}:=\alpha$ for $\alpha\in R^{im}$ and $\alpha^{\vee}:=\dfrac{2\alpha}{(\alpha,\alpha)}$ for $\alpha\in R^{re}$. From this, if $\alpha\in R^{re}$ then $(\alpha,\alpha^{\vee})=2$ and if $\alpha\in R^{im}$ then $(\alpha,\alpha^{\vee})=0$.
	\begin{lemma}\label{lem0}
		If $\al\in R^{im}$, and $k\al\in R$,  then $k=\pm 1$.
	\end{lemma}
	\begin{proof}
		For $\be\in R^{re}$ we have $(c\al,\be^{\vee})\in\Z$ implies that $c\in\Z$.
		Suppose that $2\al\in R$ and $2\al-\al\in R$, by (3) we have $(2\al,\al)\neq0$. Which contradict $\al\in R^{im}$. By induction we can show no other integral multiples are not in $R$.
	\end{proof}
	\begin{lemma}\label{lem4}
	If $\al\in R^{re}$ and $k\al\in R$, then $k\in\{\pm1/2,\pm1,\pm2\}$.
	\end{lemma}
	\begin{proof}
	Since $k\al\in R^{re}$ then $(k\al,\al^{\vee})=\dfrac{2(k\al,\al)}{(\al,\al)}=2k\in\Z$ implies $k$ is integral multiple of 1/2 and $(\al,(k\al)^{\vee})=\dfrac{2(\al,k\al)}{(k\al,k\al)}=2/k\in\Z.$ All together this implies $k\in\{\pm1/2,\pm1,\pm2\}$.
	\end{proof}
	\begin{rmk}
	Following the lemma, if $\alpha$ is a root then $2\alpha$ can be a root. Define $R_{\bar 1}=R^{im}\cup\{\alpha\in R^{re}:2\alpha\in R\}$ and $R_{\bar 0}=R\backslash R_{\bar 1}.$ Clearly $\alpha \in R_{\bar 0}$ implies $\alpha \in R^{re}$ having property that $2\alpha \notin R.$ We have $R=R_{\bar0}\cup R_{\bar 1}.$ Whenever $\al\in R_{\bar j}$, we define degree $\overline{\al}=j$ for $j\in\Z_2.$ For $\al\in R_{\bar 0}$ ($\al\in R_{\bar 1}$) we call them even (respectively odd) roots. Also we consider $\overline{\al\pm\be}:=\overline{\al}+\overline{\be}$ whenever $\al,\be,\al\pm\be\in R$.
	\end{rmk}
	A base for a GRS is a subset $\Pi=\{\alpha_1,\alpha_2,\ldots,\alpha_n\}$ of $R$ such that any $\alpha\in R$ can be uniquely represented as a linear combination $\alpha=k_1\al_1+k_2\al_2+\cdots+k_n\alpha_n$ such that either all $k_i\in\Z_{\geq 0}$ or all $k_i\in\Z_{\leq 0}$; correspondingly, we have a partition $R=R^+\cup R^-$ where $R^+=R\cap\Z_{\geq 0}-span~\Pi$ is set of positive roots and $R^-=-R^+$ is set of negative roots. Following this we have a partial order $\preceq$ on $V$ by setting $\la\preceq\mu$ if $\mu-\la\in R^+$.
	
	A generalized root system $R$ is called reducible if it can be represented as a direct orthogonal sum of two nonempty generalized root systems $R_1$ and $R_2$ : $V = V_1 \oplus V_2$, $R_i\subseteq V_i$ and $R = R_1 \cup R_2$. Otherwise the system is called irreducible. By the classification of irreducible GRS, we have a complete list of irreducible GRS (\cite{serganova1996generalizations}, Theorem 5.10). Throughout this article, we let $R$ be an irreducible GRS with isotropic roots, also we let a distinguished base $\Pi$ for these root systems be given in \cite{frappat1989structure,iohara2001central,kac1977lie}.
	
	Let $I=\{1,2,\ldots,n\}$ be the index set of $\Pi$, for $i\in I$, we denote $\bar{i}=\overline{\al_i}.$ For $\alpha=\sum\limits_{i=1}^{n}k_i\alpha_i\in R$ define the height $ht(\alpha)=\sum\limits_{i=1}^{n}k_i$. Also for $\alpha\in R$, we define $sgn(\al)=\begin{cases}1&\text{ if }\al\in R^+\\-1&\text{ if }\al\in R^-\end{cases}$ and suppose that $(\al,\al_i)\neq 0$ for $\al\in R$ and $\al_i\in\Pi\cap R^{im}$, we let $(\al,\al_i):=sgn(\al)$. 
	
\subsection{Root strings.}	Let $\al,\be\in R$, we can uniquely define two integers $r,q\geq0$ by the conditions that $$\be-r\al,\ldots,\be-\al,\be,\be+\al,\ldots,\be+q\al\in R$$ and $\be-(r+1)\al\notin R$, $\be+(q+1)\al\notin R$. The above sequence of roots is called the $\al$-string through $\be$ and we denote $m_{\al}^+(\be)=q+1$ and $m_{\al}^-(\be)=r+1$. We also write $m_i^{\pm}(\be):=m_{\al_i}^{\pm}(\be)$ with $i\in I$. If $\al\in R^{re}$ and $\be\in R$ then the root string property $(\be,\al^{\vee})=m_{\al}^-(\be)-m_{\al}^+(\be)$ holds (\cite{yousofzadeh2017locally}, Lemma 3.12).

	We now collect some results which will be useful in the sequel.
	\begin{lemma}[\cite{serganova1996generalizations}, Lemma 1.8]\label{lem2} Let $\al\in R^{im},\be\in R^{re}$ such that $(\be,\al)\neq 0$.
		\begin{enumerate}
			\item if $r_{\al}(\be)=\be\pm\al\in R^{im}$ then $(\al,\be^{\vee})=\mp1$,
			\item if $r_{\al}(\be)=\be\pm\al\in R^{re}$, then $(\al,\be^{\vee})=\mp2$.
		\end{enumerate}
	\end{lemma}
	\begin{proof} For (a) use the fact $(\be\pm\al,\be\pm\al)=0$ which implies that $(\be,\be)\pm2(\be,\al)=0$ and so $(\al,\be^{\vee})=\mp1$. 
		
		For (b) we know $(\al,(\al+\be)^{\vee})=\frac{2(\al,\be)}{(\be,\be)+2(\al,\be)}=\frac{(\al,\be^{\vee})}{1+(\al,\be^{\vee})}\in\Z$	only if $(\al,\be^{\vee})=-2$. Similarly if $\be-\al\in R^{re}$ implies that $(\al,\be^{\vee})=2$.
	\end{proof}
From the Lemma for a $\be$-string through $\al$ we have the following immediate bounds.
    \begin{cor}\label{cor1}
		For $\al\in R^{im},\be\in R^{re}$ with $(\be,\al)\neq0,$ we have
		\begin{itemize}
			\item[(a)] if $\be+\al\in R^{im}$, then $m_{\be}^+(\al)=2$,
			\item [(b)] if $\be+\al\in R^{re}$, then $m_{\be}^+(\al)=3$.
		\end{itemize}
		Similarly, $m_{\be}^-(\al)=2$ if $\be-\al\in R^{im}$ and $m_{\be}^-(\al)=3$ if $\be-\al\in R^{re}$.
	\end{cor}
	\begin{proof}
		(a) If $\al+\be\in R^{im}$ then we have $m_{\be}^-(\al)=1$ as $\be-\al\notin R$ and by above lemma $(\al,\be^{\vee})=-1$. Hence by root string property we have $-1=(\al,\be^{\vee})=m_{\be}^-(\al)-m_{\be}^+(\al)=1-m_{\be}^+(\al)$ and so $m_{\be}^+(\al)=2$. Thus $\{\al,\al+\be\}$ be the root string of $\be$ through $\al$.
		
		(b) If $\al+\be\in R^{re}$ then by above lemma $(\al,\be^{\vee})=-2$. Hence by root string property we have $-2=(\al,\be^{\vee})=m_{\be}^-(\al)-m_{\be}^+(\al)=1-m_{\be}^+(\al)$ and so $m_{\be}^+(\al)=3$. Thus $\be$-string through $\al$ be $\{\al,\al+\be,\al+2\be\}$.
	\end{proof}
	\begin{nt}\label{nt1}
		By the above corollary, suppose that $\be+\al\in R^{re}$, then $r_{\be}(\al)=\al-(\al,\be^{\vee})\be=\al-(m_{\be}^-(\al)-m_{\be}^+(\al))\be=\al-(1-3)\be=\al+2\be\in R^{im}$ as $r_{\be}$ preserves the bilinear inner product. Similarly, $\al-2\be\in R^{im}$ if $\al-\be\in R^{re}$. Also if $\be\in R_{\bar{0}}$, then $\be\pm\al\in R_{\bar{1}}^{re}$ and so $2(\be\pm\al)\in R$.
	\end{nt}
	\begin{lemma}[\cite{serganova1996generalizations}, Lemma 1.9]\label{lem1} Let $\al\in R^{im}, \be\in R$ and $(\beta,\al)\neq 0$, then $(r_{\al}(\beta),\al)\neq 0$. Further if $\be\in R^{im}$, then $(r_{\al}(\beta),r_{\al}(\beta))\neq 0$ and $(r_{\al}(\be),\be)\neq0$.
	\end{lemma}
	\begin{proof} Since $r_{\al}(\be)=\be\pm\al$, $(r_{\al}(\beta),\al)=(\be,\al)\pm(\al,\al)\neq 0$. For $\be\in R^{im}$, we have $(\be\pm\al,\be\pm\al)=\pm2(\al,\be)\neq 0$ and $(r_{\al}(\beta),\be)=(\be,\be)\pm(\al,\be)\neq 0$.
	\end{proof}	
	
	\begin{lemma}[\cite{serganova1996generalizations}, Lemma 1.10]\label{lem3} Let $\al,\be\in R^{im}$ and $(\al,\be)\neq 0$, then $\be+k\al\in R$ (with $k\neq 0$) only if $k=\pm 1$.
	\end{lemma}
	\begin{proof} Since $(\be+k\al,\be+k\al)=2k(\al,\be)\neq 0$ therefore $\be+k\al\in R^{re}$. Then $(\al,(\be+k\al)^{\vee})=\dfrac{2(\al,\be)}{2k(\al,\be)}=\dfrac{1}{k}\in\Z$ only if $k=\pm 1$.
	\end{proof}
	\begin{nt}
		By Lemma \ref{lem3} we have $m_{\be}^\pm(\al)=m_{\al}^{\pm}(\be)\leq 2$, whenever $\al,\be\in R^{im}$ with $(\al,\be)\neq 0$. Also, only one of $\be\pm\al=r_{\al}(\be)$ is the root.
	\end{nt}
    In the following Lemma we have computed the bound for the $\alpha$-string through $\beta.$ 
\begin{lemma}\label{lem5} Let $\al\in R^{im},\be\in R^{re}$ and $(\al,\be)\neq0$. 
		\begin{enumerate}
			\item if $r_{\al}(\be)\in R^{im}$, then $m_{\al}^{\pm}(\be)\leq 2$.
			\item if $r_{\al}(\be)\in R^{re}$, then $m_{\al}^{\pm}(\be)\leq 2$.
		\end{enumerate}
	\end{lemma}
	\begin{proof} (a) By Lemma \ref{lem1}, we have $(r_{\al}(\be),\al)\neq0$. If $r_{\al}(\be)\in R^{im}$, then by Lemma \ref{lem3}, we have $r_{\al}(\be)+k\al\in R$ only if $k=\pm1$. But $r_{\al}^2=id$ (\cite{serganova1996generalizations}, Lemma 1.11), we have $r_{\al}(\be)\pm\al=r_{\al}(r_{\al}(\be))=\be$. Hence $\be\pm2\al\notin R$ and so  $m_{\al}^{\pm}(\be)\leq 2$.
		
		(b) Suppose $\be+2\al\in R$ and as $\be+\al\in R$ we have two cases.
		
		Case(i): If $\be+\al\in R_{\bar 1}$, then $2(\be+\al)\in R$ and let $\gamma=\be+2\al$. Consider the $\be$-string through $\gamma$, $\{\gamma-r\be,\ldots,\gamma,\ldots,\gamma+q\be\}$, as $\be\in R^{re}$ which has root string property, we have $r-q=(\gamma,\be^{\vee})$ with $q\geq1$. Since $\gamma-\be=2\al\notin R$ (by Lemma \ref{lem0}) implies that $r=0$ and so $q=-(\gamma,\be^{\vee})=-(\be,\be^{\vee})-2(\al,\be^{\vee})=2$ as $(\al,\be^{\vee})=-2$. Hence $\gamma+2\be=2\al+3\be\in R$.
		
		Now, $(\be,\al)\neq 0$ and $3(\al+\be)\notin R$ implies that $(2\al+3\be)-\al=\al+3\be\in R$, consider the $\be$-string through $\al$, $\{\al,\al+\be,\al+2\be,\al+3\be,\ldots,\al+q\be\}$ and $q=-(\al,\be^{\vee})\geq 3$. Which contradicts $(\al,\be^{\vee})=-2$ and so $\be+\al\notin R_{\bar 1}$.
		
		Case(ii): If $\be+\al\in R_{\bar 0}$, then the $\be$-string through $\be+2\al$ is just $\{\be+2\al\}$, we have $(\be+2\al,\be^{\vee})=0$ implies that $(\be+2\al,\be)=0$ and so $(\be+\al,\be+\al)=(\be,\be)+2(\be,\al)=0$. Which is contradiction to $\be+\al\in R^{re}$. So that $\be+\al\notin R_{\bar 0}$, hence $\be+2\al\notin R$.
	\end{proof}
	
	\begin{lemma}\label{lem6} Let $\al\in R$ and $i,j\in I,i\neq j$. Assume that $\al\notin\{\pm\al_i,\pm\al_j\}$ and $\al+\al_i-\al_j\in R$. Then $\al+\al_i\in R$ if and only if $\al-\al_j\in R$. (Note that the assumption implies that $\al+\al_i\neq\al_j$ and $\al-\al_j\neq-\al_i$.)
	\end{lemma}
	\begin{proof}
		We prove this lemma by using a quiver diagram (refer, sec. \ref{sec5})  for each of the GRS. Here, we express only the positive root system in terms of the distinguished base that has been given \cite{frappat1989structure,iohara2001central,kac1977lie}. We arrange each root in a positive root system by its height (negative roots also satisfy the diagram), and we observe that each complete square from the quiver diagram gives the required result.
		\begin{center}
			{
				\begin{tikzpicture}[node distance=2cm and 2cm, >=Stealth]
					\tikzset{edge label/.style={font=\tiny}}
					
					\node (a1) at (0,0) {$\al$};
					\node (a2) [right=of a1] {$\al+\al_i$};
					
					\draw[->] (a1) -- (a2) node[midway, above,font=\tiny]{$\al_i$};
					
					\node (b1) [below=of a1] {$\al-\al_j$};
					\node (b2) [below=of a2] {$\al+\al_i-\al_j$};
										
					\draw[->] (b1) -- (b2) node[midway, below, font=\tiny]{$\al_i$};
										
					\draw[->] (a1) -- (b1) node[midway, left,font=\tiny]{$-\al_j$};
					\draw[->] (a2) -- (b2) node[midway, left,font=\tiny]{$-\al_j$};
										
					\draw[->] (a1) -- (b2) node[midway, rotate=320, below,font=\tiny]{$\al_i-\al_j$};
				
				\end{tikzpicture}
			}	
		\end{center}
			\end{proof}
			
	\begin{cor}
	Let $\al\in R^{re}$ and $\al_i,\al_j\in R^{re}$ in (\ref{lem6}), then not both $\al+\al_i$ and $\al-\al_j$ are in $R^{im}$.
	\end{cor}
	\begin{proof}
	Suppose that both $\al+\al_i$ and $\al-\al_j$ are in $R^{im}$, then $\al+\al_i+\al_j\notin R$ as $\al+\al_i-\al_j\in R$. Also, $\al+\al_j\notin R$ otherwise by above lemma, since $(\al+\al_j)+\al_i-\al_j\in R$ and $(\al+\al_j)-\al_j\in R$ implies $(\al+\al_j)+\al_i\in R$ which contradicts that $\al+\al_j+\al_i\notin R$. Now, $\al-\al_j\in R^{im}$ and $(\al-\al_j)+\al_j\in R^{re}$ then by (cor. \ref{cor1}) we have $m_j^+(\al-\al_j)=3$ implies that $\al+\al_j\in R$ becomes a contradiction. By similar arguments, we arrive a contradiction to $\al-\al_i\notin R$. Hence, not both $\al+\al_i$ and $\al-\al_j$ are in $R^{im}$.
	\end{proof}
	\begin{cor}
	Let $\al\in R_{\bar{1}}$ and $\al_i,\al_j\in R^{re}$ in (\ref{lem6}), then not both $\al+\al_i\text{ and }\al-\al_j$ are in $R^{re}$.
	\end{cor}
	\begin{proof}
	By explicit verification, the only irreducible GRS with non-empty $R_{\bar{1}}^{re}$ are $B(m,n)$ and $G(3)$. Hence, the result is clear in other root systems. In $G(3)$, upto $W$-equivalence there is unique system of simple roots $\{\al_1=\del+\ep_3,\al_2=\ep_1,\al_3=\ep_2-\ep_1\}$ and it is easy to check that there is no $\al\in R_{\bar{1}}$ such that $\al\pm\al_2\mp\al_3\in R$. Also by observations the result holds in case of $B(m,n)$.
	\end{proof}
	\begin{lemma}\label{lem12}
		 Let $\al\in R_{\bar{1}}$ and $\al_i,\al_j\in R^{re}$ in (\ref{lem6}), then $m_i^-(\al)m_j^+(\al+\al_i)=m_j^+(\al)m_i^-(\al-\al_j)$.
	\end{lemma}
	\begin{proof}
		If $\al\in R^{im}$, then $m_i^-(\al)=m_j^+(\al)=1$ and it is enough to show $m_i^-(\al-\al_j)=m_j^+(\al+\al_i)$. Suppose both $\al+\al_i\text{ and }\al-\al_j$ are in $R^{im}$, then $m_i^-(\al-\al_j)=m_j^+(\al+\al_i)=1$ as $\al-\al_j+\al_i\in R$. If one of $\al+\al_i\text{ or }\al-\al_j$ (take $\al+\al_i$) is in $R^{re}$, then by (cor. \ref{cor1}, note \ref{nt1}) we have $\al+2\al_i\in R^{im}$ and  $m_i^-(\al-\al_j)=1$ as $\al-\al_j\in R^{im}$. Hence we have to show $m_j^+(\al+\al_i)=1$, if not, that is $\al+\al_i+\al_j\in R$ (take it as $\be$). Then by $\be+\al_i-\al_j=\al+2\al_i\in R$ and $\be-\al_j=\al+\al_i\in R$ we have $\be+\al_i=\al+\al_j+2\al_i\in R$. 
		
		Note that $r_{\al_i}(\al-\al_j)=\al-\al_j+\al_i\in R^{im}\text{ as }r_{\al_i}\text{ preserves }(~,~)$, let $\ga=\al+\al_i-\al_j\in R^{im}$ and $\ga+\al_j=\al+\al_i\in R^{re}$, again by (cor. \ref{cor1}, note \ref{nt1}) we have $\ga+2\al_j=\be\in R^{im}$. We have all the thinks to proceed further to the equation, $$(r_{\al_i}(\al),\al_j^{\vee})=(\al,\al_j^{\vee})-(\al,\al_i^{\vee})(\al_i,\al_j^{\vee}).$$ 
		Since $(\al,\al_j^{\vee})=1$, $(\al,\al_i^{\vee})=-2$ and $(\al_i,\al_j^{\vee})=1-m_j^+(\al_i)$. Also $r_{\al_i}(\al)=\al+2\al_i$ and so $(\al+2\al_i,\al_j^{\vee})=m_j^-(\al+2\al_i)-m_j^+(\al+2\al_i)=-1$ as $\al+2\al_i\in R^{im}$. Then by the equation $m_j^+(\al_i)=2$  is a contradiction to $\al_i+\al_j\notin R$, hence we have $m_j^+(\al+\al_i)=1$ and the relation holds. Similarly, if we choose $\al-\al_j\in R^{re}$ (and $\al+\al_i\in R^{im}$) by following the same arguments  we have the result.
		
		Now, if $\al\in R_{\bar{1}}^{re}$ then one of $\{\al+\al_i,\al-\al_j\}$ is in $R^{re}$ and other one is in $R^{im}$. Let $\al+\al_i\in R^{im}$, then $\al+\al_i+\al_j\notin R$ which implies $\al+\al_j\notin R$. Thus, we have $m_j^+(\al)=1$ and $m_j^+(\al+\al_i)=1$. Also, $\al+\al_i\in R^{im}$ and $(\al+\al_i)-\al_i\in R^{re}$ implies that $m_i^-(\al+\al_i)=3$ and so $m_i^-(\al)=2$ (that is, $\al-\al_i\in R$). Now $(\al-\al_i)+\al_i-\al_j\in R$ and $(\al-\al_i)+\al_i\in R$ implies that $(\al-\al_i)-\al_j\in R$ (that is, $m_i^-(\al-\al_j)\geq2$). Since $(\al-\al_j)+\al_i\in R$, $\al-\al_j\in R^{re}$ and $(\al-\al_j)-\al_i\in R$, it is enough to show that $\al-\al_j+\al_i\in R^{im}$. Suppose not, since $\al+\al_i\in R^{im}$ implies that $\al+\al_i-2\al_j\in R$ and so $\{\al-2\al_j,\al-2\al_j-\al_i\}\subset R$. Since $\al-\al_j\in R_{\bar{1}}^{re}$ and $\al-\al_j+\al_i\in R^{re}$ implies that $\al-2\al_j\in R^{im}$. But we have $\{(\al-2\al_j)-\al_i,(\al-2\al_j)+\al_i\}\subset R$, which is contradiction and so $\al-\al_j+\al_i\in R^{im}$ implies that $m_i^-(\al-\al_j)=2$. Hence the relation holds by the similar arguments if we choose $\al-\al_j\in R^{im}$.
	\end{proof}
	\begin{lemma}\label{lem13}
		Let $\al\in R^{re},\al_i\in R^{re}\text{ and }\al_j\in R^{im}$ in (\ref{lem6}), then $m_i^-(\al-\al_j)=m_i^-(\al)$.
	\end{lemma}
	\begin{proof}
		Let $\al\in R_{\bar{0}}$ and if $\al-\al_j\in R^{im}$, then $m_i^-(\al-\al_j)=1$ as $\al-\al_j+\al_i\in R$, also $m_i^-(\al)=1$. Otherwise, $(\al-\al_i)+\al_i\in R$ and $(\al-\al_i)+\al_i-\al_j\in R$ implies that $\al-\al_i-\al_j\in R$ which contradict $m_i^-(\al-\al_j)=1$. For the remaining cases, we can observe from the quiver diagram that $m_i^-(\al-\al_j)=m_i^-(\al)=2$.
	\end{proof}
	\begin{lemma}\label{lem14}
		Let $\al\in R^{im},\al_i\in R^{re}\text{ and }\al_j\in R^{im}$ in (\ref{lem6}), then $m_i^-(\al-\al_j)=1$.
	\end{lemma}
    \begin{proof}
        This result hold by observation of quiver diagram.
    \end{proof}
	\section{Generators and Relations}\label{sec3}
	Following the notations from previous section, in this section we going to define our main recipe. Let $R=R_{\bar 0}\cup R_{\bar1}=R^{re}\cup R^{im}$ be a irreducible GRS with distinguished  base $\Pi=\{\al_1,\al_2,\ldots,\al_n\}$. Consider the vector super space $M$ with homogeneous basis $\{u_i:i\in I\}\cup\{v_{\alpha}:\alpha\in R\}$, the parity is given by $\overline u_i=\bar 0$ and $\overline v_{\alpha}=\overline \alpha$. For $i\in I$ define the operators $e_i,f_i\text{ and }h_i$ on $M$ by, for $j\in I,\alpha\in R$, 
	\begin{align*}
		&e_i(u_j)=|(\alpha_i,\alpha_j^{\vee})| v_{{\alpha}_i},
		&e_i(v_{\alpha})=
		\begin{cases}
			0&\text{ ; if }\alpha+\alpha_i\notin R\\(-1)^{\bar{i}}u_i&\text{ ; if }\alpha=-\alpha_i\\
			((1-\bar{i})(sgn~\al)m_i^-(\alpha)+\bar{i}~m_i^-(\alpha))v_{\alpha+\alpha_i}&\text{ ; if }\alpha\in R^{re}\\((1-\bar{i})~m_i^-(\al)+\bar{i}~(\alpha,\alpha_i^{\vee}))(sgn~\al)v_{\alpha+\alpha_i}&\text{ ; if }\alpha\in R^{im}
		\end{cases}\\
		&f_i(u_j)=|(\alpha_i,\alpha_j^{\vee})|v_{-\alpha_i},
		&f_i(v_{\alpha})=\begin{cases}
			0&\text{ ; if }\alpha-\alpha_i\notin R\\u_i&\text{ ; if }\alpha=\alpha_i\\
			(sgn~\al)m_i^+(\alpha)v_{\alpha-\alpha_i}&\text{ ; if }\alpha\in R^{re}\\((1-\bar{i})(sgn~\al)m_i^+(\al)+\bar{i}~(\alpha,\alpha_i^{\vee}))v_{\alpha-\alpha_i}&\text{ ; if }\alpha\in R^{im}
		\end{cases}\\
		&h_i(u_j)=0\text{ and } h_i(v_\alpha)=(\alpha,\alpha_i^{\vee})v_{\alpha}.
	\end{align*} 
	\begin{lemma}\label{lem7} The following results are holds for the above defined operators and we simply write $gf$ instead of composition of the maps $g\circ f$.
		\begin{enumerate}
			\item $e_i,f_i$ are homogeneous linear map of degree $\overline e_i=\overline f_i=\bar i$ and $h_i$ is homogeneous linear map of degree $\overline h_i= 0.$
			\item $h_ih_j-h_jh_i=0$.
			\item $h_ie_j-e_jh_i=(\alpha_j,\alpha_i^{\vee})e_j$.
			\item $h_if_j-f_j h_i=-(\alpha_j,\alpha_i^{\vee})f_j$.
			\item $e_i f_i-(-1)^{\bar i}f_i e_i=h_i$.
			\item For $i\neq j$, $e_i f_j-(-1)^{\bar i\bar j}f_j e_i=0$.
		\end{enumerate}
		\end{lemma}
	\begin{proof} (a) Since $\overline{\alpha\pm\beta}=\overline\alpha+\overline\beta$ and by the definition of $e_i$ we have $u_j\mapsto v_{\alpha_i}, v_{\alpha}\mapsto v_{\alpha+\alpha_i}$ and $v_{-\alpha_i}\mapsto u_i$. Similarly, by the definition of $f_i$ we have $u_j\mapsto v_{-\alpha_i}, v_{\alpha}\mapsto v_{\alpha-\alpha_i}$ and $v_{\alpha_i}\mapsto u_i$. Hence $e_i$'s and $f_i$'s are of degree $\bar i$. Also, $h_i$ maps $u_i\mapsto 0$ and $v_{\alpha}\mapsto v_{\alpha}$. Hence $h_i$'s are of degree 0.
		
		\noindent
		(b) Since, $h_j h_i(u_k)=0=h_i h_j(u_k)$ and $h_j h_i(v_{\alpha})=(\alpha,\alpha_i^{\vee})(\alpha,\alpha_j^{\vee})v_{\alpha}=h_i h_j(v_{\alpha})$.
		
		\noindent
		(c) $(h_i e_j-e_j h_i)(u_i)=h_i(|(\alpha_j,\alpha_i^{\vee})|v_{\alpha_j})=|(\alpha_j,\alpha_i^{\vee})|(\alpha_j,\alpha_i^{\vee})v_{\alpha_j}=(\alpha_j,\alpha_i^{\vee})e_j(u_i)$ and 
		\begin{eqnarray*}
			h_i(e_j(v_{\alpha}))&=&h_i\left(\begin{cases}
				0&\text{ ; if }\alpha+\alpha_j\notin R\\(-1)^{\bar{j}}u_j&\text{ ; if }\alpha=-\alpha_j\\
				((1-\bar{j})(sgn~\al)m_j^-(\alpha)+\bar{j}~m_j^-(\alpha))v_{\alpha+\alpha_j}&\text{ ; if }\alpha\in R^{re}\\((1-\bar{j})m_j^-(\al)+\bar{j}(\alpha,\alpha_j^{\vee}))(sgn~\al)v_{\alpha+\alpha_j}&\text{ ; if }\alpha\in R^{im}
			\end{cases}\right)\\
			&=&\begin{cases}
				0&\text{; }\alpha+\alpha_j\notin R\\0&\text{; }\alpha=-\alpha_j\\
				((1-\bar{j})(sgn~\al)m_j^-(\alpha)+\bar{j}~m_j^-(\alpha))(\alpha+\alpha_j,\alpha_i^{\vee}) v_{\alpha+\alpha_j}&\text{; }\alpha\in R^{re}\\((1-\bar{j})m_j^-(\al)+\bar{j}(\alpha,\alpha_j^{\vee}))(sgn~\al)(\alpha+\alpha_j,\alpha_i^{\vee})v_{\alpha+\alpha_j}&\text{; }\alpha\in R^{im}
			\end{cases}	
		\end{eqnarray*}

		\noindent	
		$e_j(h_i(v_{\alpha}))=e_j((\alpha,\alpha_i^{\vee}) v_{\alpha})=
		(\alpha,\alpha_i^{\vee})\begin{cases}
			0&\text{ ; if }\alpha+\alpha_j\notin R\\(-1)^{\bar{j}}u_j&\text{ ; if }\alpha=-\alpha_j\\
			((1-\bar{j})(sgn~\al)m_j^-(\alpha)+\bar{j}~m_j^-(\alpha))v_{\alpha+\alpha_j}&\text{ ; if }\alpha\in R^{re}\\((1-\bar{j})m_j^-(\al)+\bar{j}(\alpha,\alpha_j^{\vee}))(sgn~\al)v_{\alpha+\alpha_j}&\text{ ; if }\alpha\in R^{im}
		\end{cases}$.
		
		\noindent
		\begin{eqnarray*}
			(h_i e_j-e_j h_i)(v_{\alpha})&=&\begin{cases}
				0&\text{; }\alpha+\alpha_j\notin R\\-(-1)^{\bar j}(\alpha,\alpha_i^{\vee})u_j&\text{; }\alpha=-\alpha_j\\(\alpha_j,\alpha_i^{\vee})((1-\bar{j})(sgn~\al)
				m_j^-(\alpha)+\bar{j}
				m_j^-(\alpha))v_{\alpha+\alpha_j}&\text{; }\alpha\in R^{re}\\(\alpha_j,\alpha_i^{\vee})((1-\bar{j})m_j^-(\al)+\bar{j}(\alpha,\alpha_j^{\vee}))(sgn~\al)v_{\alpha+\alpha_j}&\text{; }\alpha\in R^{im}
			\end{cases}\\
			&=&(\alpha_j,\alpha_i^{\vee}) e_j(v_{\alpha}).
		\end{eqnarray*}
		
		\noindent
		(d) Similar to the proof of (c).
		
		\noindent
		(e) Since, $e_if_i(u_k)=(-1)^{\bar i}|(\alpha_i,\alpha_k^{\vee})|u_i$ and $f_ie_i(u_k)=|( \alpha_i,\alpha_k^{\vee})|u_i$. Hence 
		$(e_if_i-(-1)^{\bar i}f_ie_i)(u_k)=0$. Now, at $v_{\alpha}$, values can be identified by all possible cases.
		
		Case (i): For $\al=\al_i$, $e_if_i(v_{\al_i})=|(\al_i,\al_i^{\vee})|v_{\al_i}$. Since $2\al_i\notin R$ as $\al_i\in\Pi$, hence $e_i(v_{\alpha_i})=0$ implies $f_ie_i(v_{\alpha_i})=0$. If $\al_i\in R^{re}~(\bar i=0)$, $e_if_i(v_{\al_i})=2v_{\al_i}$ and $h_i(v_{\al_i})=(\al_i,\al_i^{\vee})v_{\al_i}=2v_{\al_i}$. If $\al_i\in R^{im}$, $e_if_i(v_{\al_i})=0$ and $h_i(v_{\al_i})=0$, hence the relation holds.
		
		Case (ii): For $\al=-\al_i$, $e_i(v_{-\al_i})=(-1)^{\bar i}u_i$ implies that $f_ie_i(v_{-\al_i})=(-1)^{\bar i}|(\al_i,\al_i^{\vee})|v_{-\al_i}$. Again $2\al_i\notin R$ as $\al_i\in\Pi$, hence $f_i(v_{-\alpha_i})=0$ implies that $e_if_i(v_{-\alpha_i})=0$. If $\al_i\in R^{re}~(\bar i=0)$, the relation holds as $h_i(v_{-\al_i})=-(\al_i,\al_i^{\vee})v_{-\al_i}$ and if $\al_i\in R^{im}$ then $f_ie_i(v_{-\al_i})=0$ also $h_i(v_{-\al_i})=0$, the relation holds.
		
		Case (iii): For $\al\neq\pm\al_i$, suppose that $\al+\al_i\notin R$ and $\al-\al_i\notin R$, then $e_i(v_{\al})=f_i(v_{\al})=0$ and $h_i(v_{\al})=(\al,\al_i^{\vee})v_{\al}=0$ as if $\al_i\in R^{re}$ which has root string property with $m_i^{\pm}(\al)=1$ and if $\al_i\in R^{im}$, since $\al\pm\al_i\notin R$ implies that $(\al,\al_i)=0$.
		
		If $\al+\al_i\in R$ but $\al-\al_i\notin R$, then $e_if_i(v_{\al})=0$ and $m_i^-(\al)=1$. For $\bar{i}=0,~(\al_i\in R^{re})$,
		$$e_i(v_\al)=(sgn~\al)m_i^-(\al)v_{\al+\al_i}$$ $$f_ie_i(v_{\al})=(sgn~\al)m_i^-(\al)(sgn~\al+\al_i)m_i^+(\al+\al_i)v_{\al}.$$
		Since, $sgn~\al=sgn(\al+\al_i)$ and $m_i^+(\al+\al_i)=m_i^+(\al)-1$, therefore $-(-1)^{\bar i}f_ie_i(v_{\al})=-(m_i^+(\al)-1)v_{\al}$ and $h_i(v_\al)=(\al,\al_i^{\vee})v_{\al}=(m_i^-{(\al)}-m_i^+(\al))v_{\al}=(1-m_i^+(\al))v_{\al}$. Hence equality holds.
		
		For $\bar{i}=1,(\al_i\in R^{im})$, we have
		\begin{align*}
			e_i(v_{\alpha})&=\begin{cases}
				m_i^-(\alpha)v_{\al+\alpha_i}& \text{; if }\alpha\in R^{re}\\(sgn~\al)(\alpha,\alpha_i^{\vee})v_{\al+\al_i}&\text{; if }\alpha\in R^{im}
			\end{cases}\\
			f_ie_i(v_{\alpha})&=\begin{cases}
				m_i^-(\alpha)(sgn~\al+\al_i)m_i^+(\al+\al_i)v_{\alpha}& \text{; if }\al,\al+\al_i\in R^{re}\\
				m_i^-(\al)(\al+\al_i,\al_i^{\vee})v_{\al}&\text{; if }\al\in R^{re}\text{ and }\al+\al_i\in R^{im}\\
				(sgn~\al)(\al,\al_i^{\vee})(sgn~\al+\al_i)m_i^+(\al+\al_i)v_{\al}&\text{; if }\al\in R^{im}.
			\end{cases}
		\end{align*}		
		If $\al,\al+\al_i\in R^{re}$, by Lemma (\ref{lem5},(b)) $m_i^+(\al)=2$, then $-(-1)^{\bar i}f_ie_i(v_{\al})=(sgn~\al+\al_i)(m_i^+(\al)-1)v_{\al}=(sgn~\al+\al_i)v_{\al}=h_i(v_{\al})$ as $(\al,\al_i^{\vee})=(sgn~\al)=sgn(\al+\al_i)$.
		
		In second alternate,  $f_ie_i(v_{\al})=(\al+\al_i,\al_i)v_{\al}=(\al,\al_i)v_{\al}=h_i(v_{\al})$.
		
		In last one $\al~\&~\al_i\in R^{im}$ implies that $\al+\al_i\in R^{re}$ (by Lemma \ref{lem1}), $sgn(\al)=sgn(\al+\al_i)$ and $m_i^+(\al+\al_i)=1$ (by Lemma \ref{lem3}) implies that $f_ie_i(v_{\al})=(\al,\al_i)v_{\al}=h_i(v_{\al})$.
		
		If $\al-\al_i\in R$ but $\al+\al_i\notin R$, then $f_ie_i(v_{\al})=0$ and $m_i^+(\al)=1$. For $\bar{i}=0,~(\al_i\in R^{re})$,
		$$f_i(v_\al)=(sgn~\al)m_i^+(\al)v_{\al-\al_i}$$
		$$e_if_i(v_{\al})=(sgn~\al)m_i^+(\al)(sgn~\al-\al_i)m_i^-(\al-\al_i)v_{\al}$$
		Since, $sgn(\al)=sgn(\al-\al_i)$ and $m_i^-(\al-\al_i)=m_i^-(\al)-1$, therefore $e_if_i(v_{\al})=(m_i^-(\al)-m_i^+(\al))v_{\al}=(\al,\al_i^{\vee})v_{\al}=h_i(v_{\al})$.
		
		For $\bar{i}=1,~(\al_i\in R^{im})$, we have
		\begin{align*}
			f_i(v_{\alpha})&=\begin{cases}
				(sgn~\al)m_i^+(\alpha)v_{\alpha-\al_i}& \text{; if }\al\in R^{re}\\(\al,\al_i)v_{\al-\al_i}&\text{; if }\al\in R^{im}
			\end{cases}\\
			e_if_i(v_{\alpha})&=\begin{cases}
				(sgn~\al)m_i^+(\alpha)m_i^-(\al-\al_i)v_{\alpha}& \text{; if }\al,\al-\al_i\in R^{re}\\
				(sgn~\al)m_i^+(\al)(sgn~\al-\al_i)(\al-\al_i,\al_i)v_{\al}&\text{; if }\al\in R^{re}\text{ and }\al-\al_i\in R^{im}\\
				(\al,\al_i)m_i^-(\al-\al_i)v_{\al}&\text{; if }\al\in R^{im}.
			\end{cases}
		\end{align*}
		
		If $\al,\al-\al_i\in R^{re}$, then by Lemma (\ref{lem5},(b)) $m_i^-(\al)=2$ implies that $e_if_i(v_{\al})=(sgn~\al)(m_i^-(\al)-1)v_{\al}=(sgn~\al)v_{\al}=h_i(v_{\al})$ as $(\al,\al_i)=(sgn~\al)$.
		
		In second alternate, $sgn(\al)=sgn(\al-\al_i)$ hence $e_if_i(v_{\al})=(\al-\al_i,\al_i)v_{\al}=(\al,\al_i)v_{\al}=h_i(v_{\al})$.
		
		In last one $\al,\al_i\in R^{im}\Rightarrow\al-\al_i\in R^{re}$ (by Lemma \ref{lem1}) and $m_i^-(\al-\al_i)=1$ (by Lemma \ref{lem3}) implies the equality.
		
		If both $\al\pm\al_i\in R\Rightarrow\al_i\notin R^{im}$ and since $\Pi\cap R_{\bar 1}^{re}=\emptyset\Rightarrow\bar{\al_i}=0$. Then 
		$$e_if_i(v_{\al})=m_i^+(\al)m_i^-(\al-\al_i)v_{\al},$$ $$f_ie_i(v_{\al})=m_i^-(\al)m_i^+(\al+\al_i)v_{\al}.$$ 
		
		Hence, $(m_i^+(\al)m_i^-(\al-\al_i)-(-1)^{\bar i}m_i^-(\al)m_i^+(\al+\al_i))v_{\al}=(m_i^-(\al)-m_i^+(\al))=(\al,\al_i^{\vee})v_{\al}$.
		
		\noindent
		(f) For $i\neq j$,
		$e_if_j(u_k)=|( \alpha_j,\alpha_k^{\vee})|m_i^-(-\al_j)v_{-\alpha_j+\alpha_i}=0$ and $f_je_i(u_k)=|( \alpha_i,\alpha_k^{\vee})|m_j^+(\al_i)v_{\alpha_i-\alpha_j}=0$, since $\alpha_i-\alpha_j\notin R$. Now at $v_{\alpha}$ values can be identified by all possible cases.
		
		Case (i): For $\alpha=\alpha_i$, since $\alpha_i-\alpha_j\notin R$, $f_j(v_{\alpha_i})=0\text{ implies that } e_i(f_j(v_{\alpha_i}))=0$ and $e_i(v_{\alpha_i})=0$ as $2\al_i\notin R$ implies $f_j(e_i(v_{\alpha_i}))=0$.
		
		Case (ii): For $\alpha=-\alpha_j$, since $\alpha_i-\alpha_j\notin R$, $e_i(v_{-\alpha_j})=0\text{ implies that } f_j(e_i(v_{-\alpha_j}))=0$ and $f_j(v_{-\alpha_j})=0$ as $2\al_j\notin R$ implies $e_i(f_j(v_{-\alpha_j}))=0$.
		
		Case (iii): For $\alpha=\alpha_j$, $f_j(v_{\alpha_j})=u_j$ and $e_i(f_j(v_{\alpha_j}))=|(\alpha_i,\alpha_j^{\vee})|v_{\alpha_i}$. For $\bar{i}=0~(\al_i\in R^{re})$
		\begin{align*}
			e_i(v_{\alpha_j})&=\begin{cases}
				0&\text{; if }\al_j+\al_i\notin R\\
				(sgn~\al_j)m_i^-(\alpha_j)v_{\al_j+\alpha_i}& \text{; otherwise }\\
			\end{cases}\\
			f_je_i(v_{\alpha_j})&=\begin{cases}
				0&\text{; if }\al_j+\al_i\notin R\\
				(sgn~\al_j)m_i^-(\alpha_j)(sgn~\al_j+\al_i)m_j^+(\al_j+\al_i)v_{\alpha_i}& \text{; if }\al_j\in R^{re}\\
				(sgn~\al_j)m_i^-(\al_j)(sgn~\al_j+\al_i)m_j^+(\al_j+\al_i)v_{\al_i}&\text{; if }\al_j\in R^{im}\text{ and }\al_j+\al_i\in R^{re}\\
				(sgn~\al_j)m_i^-(\al_j)(\al_i+\al_j,\al_j^{\vee})v_{\al_i}&\text{; if }\al_j,\al_j+\al_i\in R^{im}.
			\end{cases}\\
		\end{align*}
		
		In first alternate, if $\al_j\in R^{re}$, then $(\al_i,\al_j^{\vee})=0$ as $m_j^{\pm}(\al_i)=1$ and if $\al_j\in R^{im}$ with $\al_i\pm\al_j\notin R$ implies that $(\al_j,\al_i)=0$. Hence the relation holds.
		
		In second alternate ($\al_j\in R^{re},~\bar j=0)$, $sgn(\al_j)=sgn(\al_j+\al_i)$, $m_i^-(\al_j)=1$ and $m_j^+(\al_i+\al_j)=m_j^+(\al_i)-1=m_j^+(\al_i)-m_j^-(\al_i)=-(\al_i,\al_j^{\vee})$. Then $f_je_i(v_{\al_j})=-(\al_i,\al_j^{\vee})v_{\al_i}$. As $\al_i-\al_j\notin R$ implies that $(\al_i,\al_j^{\vee})=r-q=-q<0$, hence $e_if_j(v_{\al_j})=-(\al_i,\al_j^{\vee})v_{\al_i}$ and the relation holds. 
		
		In third one ($\al_i+\al_j\in R^{re}$,~$\al_j\in R^{im},~\bar j=1$), $m_i^-(\al_j)=1$, $f_je_i(v_{\al_j})=(m_j^+(\al_i)-1)v_{\al_i}=v_{\al_i}$ as $m_j^+(\al_i)=2$. Since $\al_i\in R^+$ hence $(\al_i,\al_j)=sgn(\al_i)=1$ and so $e_if_j(v_{\al_j})=v_{\al_i}$, hence the relation holds.		
		
		In last one, $m_i^-(\al_j)=1$ and $(\al_j+\al_i,\al_j)=(\al_i,\al_j)=sgn(\al_i)=1$ as $\al_i\in R^+$. Therefore $f_je_i(v_{\al_j})=v_{\al_i}=e_if_j(v_{\al_j})$.
		
		Now for $\bar{i}=1~(\al_i\in R^{im})$,
		\begin{align*}
			e_i(v_{\alpha_j})&=\begin{cases}
				0&\text{; if }\al_j+\al_i\notin R\\
				m_i^-(\alpha_j)v_{\al_j+\alpha_i}& \text{; if }\al_j\in R^{re}\\
				(sgn~\al_j)(\al_j,\al_i^{\vee})v_{\al_j+\al_i}&\text{ if }\al_j\in R^{im}
			\end{cases}\\
			f_je_i(v_{\alpha_j})&=\begin{cases}
				0&\text{; if }\al_j+\al_i\notin R\\
				m_i^-(\alpha_j)(sgn~\al_j+\al_i)m_j^+(\al_j+\al_i)v_{\alpha_i}& \text{; if }\al_j\in R^{re}\\
				(sgn~\al_j)(\al_j,\al_i)(sgn~\al_j+\al_i)m_j^+(\al_j+\al_i)v_{\al_i}&\text{; if }\al_j\in R^{im}\\
			\end{cases}\\
		\end{align*}
		
		In first alternate, if $\al_j\in R^{re}$, then $(\al_i,\al_j^{\vee})=0$ as $m_j^{\pm}(\al_i)=1$ and since $\al_j\notin R^{im}$ as $j\neq i$, hence the relation holds.
		
		In second alternate ($\al_j\in R^{re},~\bar j=0)$, $m_i^-(\al_j)=1$ and $m_j^+(\al_i+\al_j)=m_j^+(\al_i)-1=m_j^+(\al_i)-m_j^-(\al_i)=-(\al_i,\al_j^{\vee})$. Then $f_je_i(v_{\al_j})=-(\al_i,\al_j^{\vee})v_{\al_i}$. As $\al_i-\al_j\notin R$ implies that $(\al_i,\al_j^{\vee})=r-q=-q<0$, hence $e_if_j(v_{\al_j})=-(\al_i,\al_j^{\vee})v_{\al_i}$.
		
		Last one is not a case which we have as $j\neq i$ and $\al_j\in R^{im}$.
		
		Case (iv): For $\alpha=-\alpha_i$, $e_i(v_{-\alpha_i})=(-1)^{\bar i}u_i$ and $f_j(e_i(v_{-\alpha_i}))=(-1)^{\bar i}|( \alpha_j,\alpha_i^{\vee})|v_{-\alpha_j}$. 
		
		\noindent
		If $\bar{j}=0~(\al_j\in R^{re})$, we have
		\begin{align*}
			f_j(v_{-\alpha_i})&=\begin{cases}
				0&\text{; if }\al_i+\al_j\notin R\\
				(sgn-\al_i)m_j^+(-\alpha_i)v_{-\al_i-\alpha_j}& \text{; if }\alpha_j+\al_i\in R\\
			\end{cases}\\
			e_if_j(v_{-\alpha_i})&=\begin{cases}
				0&\text{; if }\al_i+\al_j\notin R\\
				(sgn-\al_i)m_j^+(-\alpha_i)(sgn-\al_i-\al_j)m_i^-(-\al_i-\al_j)v_{-\alpha_j}& \text{; if }\al_i\in R^{re}\\
				(sgn-\al_i)m_j^+(-\al_i)m_i^-(-\al_i-\al_j)v_{-\al_j}&\text{; if }\al_i\in R^{im}\text{ and }\al_i+\al_j\in R^{re}\\
				(sgn-\al_i)m_j^+(-\al_i)(sgn-\al_i-\al_j)(-\al_i-\al_j,\al_i^{\vee})v_{-\al_j}&\text{; if }\al_i,\al_i+\al_j\in R^{im}.
			\end{cases}
		\end{align*}
		
		In first alternate, if $\al_i\in R^{re}$, then $(\al_j,\al_i^{\vee})=0$ as $m_i^{\pm}(\al_j)=1$ and if $\al_i\in R^{im}$ with $\al_j\pm\al_i\notin R$ implies that $(\al_j,\al_i)=0$.
		
		In second alternate ($\al_i\in R^{re},~\bar i=0$), $sgn(-\al_i)=sgn(-\al_i-\al_j)$, $m_j^+(-\al_i)=1$ and $m_i^-(-\al_j-\al_i)=m_i^-(-\al_j)-1=m_i^-(-\al_j)-m_i^+(-\al_j)=-(\al_j,\al_i^{\vee})$. Then $e_if_j(v_{-\al_j})=-(\al_j,\al_i^{\vee})v_{-\al_j}$ and as $\al_j-\al_i\notin R$ implies that $(\al_j,\al_i^{\vee})=r-q=-q<0$. Hence $f_je_i(v_{-\al_j})=-(\al_j,\al_i^{\vee})v_{-\al_j}$ and the relation holds.
		
		In third one ($\al_i\in R^{im},~\bar i=1$), $sgn(-\al_i)=-1$ and $m_j^+(-\al_i)=1$. Hence $e_if_j(v_{-\al_j})=-(m_i^-(-\al_j)-1)v_{-\al_j}=-v_{-\al_j}$ as $m_i^-({-\al_j})=2$ and $f_je_i(v_{-\al_i})=-v_{-\al_j}$ as $(\al_j,\al_i^{\vee})=sgn(\al_j)=1$ since $\al_j\in R^+$.
		
		In last one ($\bar i=1$), $(-\al_j-\al_i,\al_i^{\vee})=-(\al_j,\al_i^{\vee})=-sgn(\al_j)=-1$ as $\al_j\in R^+$. Therefore $e_if_j(v_{-\al_i})=-v_{-\al_j}$ and $f_je_i(v_{-\al_i})=-v_{-\al_j}$, hence the relation holds.
		
		If $\bar{j}=1~(\al_j\in R^{im})$, then 
		\begin{align*}
			f_j(v_{-\alpha_i})&=\begin{cases}
				0&\text{; if }\al_i+\al_j\notin R\\
				(sgn-\al_i)m_j^+(-\alpha_i)v_{-\al_i-\alpha_j}& \text{; if }\al_i\in R^{re}\\
				(-\al_i,\al_j^{\vee})v_{-\al_i-\al_j}&\text{; if }\al_i\in R^{im}
			\end{cases}\\
			e_if_j(v_{-\alpha_i})&=\begin{cases}
				0&\text{; if }\al_i+\al_j\notin R\\
				(sgn-\al_i)m_j^+(-\alpha_i)(sgn-\al_i-\al_j)m_i^-(-\al_i-\al_j)v_{-\alpha_j}& \text{; if }\al_i\in R^{re}\\
				(-\al_i,\al_j^{\vee})m_i^-(-\al_i-\al_j)v_{-\al_j}&\text{; if }\al_i\in R^{im}\\
			\end{cases}
		\end{align*}
		
		In first alternate, if $\al_i\in R^{re}$, then $(\al_j,\al_i^{\vee})=0$ as $m_i^{\pm}(\al_j)=1$ and since $\al_i\notin R^{im}$ as $i\neq j$.
		
		In second alternate ($\al_i\in R^{re},~\bar i=0$), $sgn(-\al_i-\al_j)=(sgn-\al_i-\al_j)$, $m_j^+(-\al_i)=1$ and $m_i^-(-\al_j-\al_i)=m_i^-(-\al_j)-1=m_i^-(-\al_j)-m_i^+(-\al_j)=-(\al_j,\al_i^{\vee})$. Then $e_if_j(v_{-\al_j})=-(\al_j,\al_i^{\vee})v_{-\al_j}$ and as $\al_j-\al_i\notin R$ implies that $(\al_j,\al_i^{\vee})=r-q=-q<0$. Hence $f_je_i(v_{-\al_j})=-(\al_j,\al_i^{\vee})v_{-\al_j}$.
		
		Last one is not a case which we have, hence the relation holds.
		
		Case (v): Consider $\alpha\notin\{\pm\alpha_i,\pm\alpha_j\}$. If both $\al+\al_i\notin R$ and $\al-\al_j\notin R$, then $e_i(v_{\al})=f_j(v_{\al})=0$. Also, if $\al+\al_i-\al_j\notin R$, the relation holds. By Lemma \ref{lem6} it remains to consider the case where $\al+\al_i\in R$, $\al-\al_j\in R$ and $\al+\al_i-\al_j\in R$ and the relation is verified for all possible cases.
		\begin{itemize}
			\item[1.] If $\al_i,\al_j\in R^{re}$, then $$e_if_j(v_{\al})=(sgn~\al)m_j^+(\al)(sgn~\al-\al_j)m_i^-(\al-\al_j)~v_{\al+\al_i-\al_j}$$ $$f_je_i(v_{\al})=(sgn~\al)m_i^-(\al)(sgn~\al+\al_i)m_j^+(\al+\al_i)~v_{\al+\al_i-\al_j}.$$
			If $\al\in R_{\bar{0}}$, then by (\cite{geck2017construction}, Lemma 2.6) we have $m_i^-(\al)m_j^+(\al+\al_i)=m_j^+(\al)m_i^-(\al-\al_j)$ and hence the relation holds. By Lemma \ref{lem12} we have the relation in the case of $\al\in R_{\bar{1}}$.
			
			\item[2.] If $\al_i\in R^{re}, \al_j\in R^{im}$ and $\al\in R^{re}$, then $$e_if_j(v_{\al})=(sgn~\al)m_j^+(\al)(sgn~\al-\al_j)m_i^-(\al-\al_j)~v_{\al+\al_i-\al_j}$$ $$f_je_i(v_{\al})=(sgn~\al)m_i^-(\al)\begin{cases}
			(sgn~\al+\al_i)m_j^+(\al+\al_i)~v_{\al+\al_i-\al_j}&\text{; if }	\al+\al_i\in R^{re}\\(\al+\al_i,\al_j)~v_{\al+\al_i-\al_j}&\text{; if }\al+\al_i\in R^{im}.
			\end{cases}$$
			Since $m_j^+(\al)=m_j^+(\al+\al_i)=1$ as $\al+\al_j$ and $\al+\al_i+\al_j$ are not in $R$, also $(\al+\al_i,\al_j)=sgn(\al+\al_i)$. Hence, by Lemma \ref{lem13} we have the relation. 
			\item[3.] If $\al_i\in R^{im},\al_j\in R^{re}$ and $\al\in R^{re}$, then 
			$$e_if_j(v_{\al})=(sgn~\al)m_j^+(\al)\begin{cases}
				m_i^-(\al-\al_j)~v_{\al+\al_i-\al_j}&\text{ ; if }\al-\al_j\in R^{re}\\(sgn~\al-\al_j)(\al-\al_j,\al_i)~v_{\al+\al_i-\al_j}&\text{ ; if }\al-\al_j\in R^{im}
			\end{cases}$$
			
			$$f_je_i(v_{\al})=m_i^-(\al)(sgn~\al+\al_i)m_j^+(\al+\al_i)~v_{\al+\al_i-\al_j}.$$
			Since $m_i^-(\al-\al_j)=m_i^-(\al)=1$ as $\al-\al_i$ and $\al-\al_j-\al_i$ are not in $R$, also $(\al-\al_j,\al_i)=sgn(\al-\al_j)$. Now, apply (2) to $-\al+\al_j-\al_i$ with the roles of $\al_i$ and $\al_j$ are exchanged. Consequently, we obtain that  $m_j^-(-\al-\al_i)=m_j^-(-\al)$. It remains to use the fact that $m_{\be}^{\pm}(\ga)=m_{\be}^{\mp}(-\ga)$ for all $\be,\ga\in R$ and so the relation holds as $m_j^+(\al)=m_j^+(\al+\al_i)$.
			
			\item[4.] If $\al_i\in R^{re},\al_j\in R^{im}$ and $\al\in R^{im}$, then 
		 $$e_if_j(v_{\al})=(\al,\al_j)(sgn~\al-\al_j)m_i^-(\al-\al_j)~v_{\al+\al_i-\al_j}$$ $$f_je_i(v_{\al})=(sgn~\al)m_i^-(\al)\begin{cases}
		 	(sgn~\al+\al_i)m_j^+(\al+\al_i)~v_{\al+\al_i-\al_j}&\text{ ; if }\al+\al_i\in R^{re}\\(\al+\al_i,\al_j)~v_{\al+\al_i-\al_j}&\text{ ; if }\al+\al_i\in R^{im}.
		 \end{cases}$$
		 Since $m_i^-(\al)=m_j^+(\al+\al_i)=1$ as $\al-\al_i$ and $\al+\al_i+\al_j$ are not in $R$, also $(\al,\al_j)=(\al+\al_i,\al_j)=sgn(\al)$. Hence, by Lemma \ref{lem14} we have the relation.
			
			\item[5.] If $\al_i\in R^{im},\al_j\in R^{re}$ and $\al\in R^{im}$, then 
			 $$e_if_j(v_{\al})=(sgn~\al)m_j^+(\al)\begin{cases}
			 	m_i^-(\al-\al_j)~v_{\al+\al_i-\al_j}&\text{ ; if }\al-\al_j\in R^{re}\\(sgn~\al-\al_j)(\al-\al_j,\al_i)~v_{\al+\al_i-\al_j}&\text{ ; if }\al-\al_j\in R^{im}
			 \end{cases}$$ $$f_je_i(v_{\al})=(sgn~\al)(\al,\al_i)(sgn~\al+\al_i)m_j^+(\al+\al_i)~v_{\al+\al_i-\al_j}.$$
			Since $\al+\al_j$ and $\al-\al_j-\al_i$ are not in $R$ which implies that $m_j^+(\al)=m_i^-(\al-\al_j)=1$, also $(\al-\al_j,\al_i)=(\al,\al_i)=sgn(\al)$. Now, apply (4) to $-\al+\al_j-\al_i$ with the roles of $\al_i$ and $\al_j$ are exchanged. Consequently, we obtain that $m_j^-(-\al-\al_i)=1=m_j^+(\al+\al_i)$ and we have the relation.
		\end{itemize}

	\end{proof}
	\begin{lemma}\label{lem8}
		The homogeneous linear map $\om:M\rightarrow M$ defined by $\om(u_i)=u_i,~\om(v_{\al})=-v_{-\al}\text{ and }\om(v_{-\al})=-(-1)^{\bar{\al}}v_{\al}$, for $\al\in R^+$, which satisfies the relations $\om e_i\om^{-1}=-f_i$, $\om f_i\om^{-1}=-(-1)^{\bar{i}}e_i$ and $\om h_i\om^{-1}=-h_i$.
	\end{lemma}
	\begin{proof}
		Since $\om$ is invertible and $\om^{-1}(u_i)=u_i,~\om^{-1}(v_{\al})=-(-1)^{\bar{\al}}v_{-\al}$ and $\om^{-1}(v_{-\al})=-v_{\al}$, for $\al\in R^+$, we have
		$$\om e_i\om^{-1}(u_j)=\om e_i(u_j)=|(\al_i,\al_j^{\vee})|\om(v_{\al_i})=-|(\al_i,\al_j^{\vee})|v_{-\al_i}=-f_i(u_j)$$ and $$\om f_i\om^{-1}(u_j)=\om f_i(u_j)=|(\al_i,\al_j^{\vee})|\om(v_{-\al_i})=-(-1)^{\bar{i}}|(\al_i,\al_j^{\vee})|v_{\al_i}=-(-1)^{\bar{i}}e_i(u_j).$$ 
		
		Also $$\om e_i\om^{-1}(v_{\al_i})=-(-1)^{\bar{i}}\om e_i(v_{-\al_i})=-(-1)^{\bar{i}}\om((-1)^{\bar{i}}u_i)=-u_i=-f_i(v_{\al_i})$$ and $$\om f_i\om^{-1}(v_{-\al_i})=-\om f_i(v_{\al_i})=-\om(u_i)=-(-1)^{\bar{i}}((-1)^{\bar{i}}u_i)=-(-1)^{\bar{i}}e_i(v_{-\al_i}).$$
		
		Now, we prove the relation case by case. For $\bar{i}=0$, $e_i(v_{\al})=(sgn~\al)m_i^-(\al)v_{\al+\al_i}$ and $f_i(v_{\al})=(sgn~\al)m_i^+(\al)v_{\al-\al_i}$. If $\al\in R^+$, then $-\al+\al_i\in-R^+$ and
		\begin{eqnarray*}
			\om e_i\om^{-1}(v_{\al})&=&-(-1)^{\bar{\al}}\om e_i(v_{-\al})\\&=&-(-1)^{\bar{\al}}(sgn-\al)m_i^-(-\al)\om(v_{-\al+\al_i})\\&=&-(-1)^{\bar{\al}}(sgn-\al)m_i^-(-\al)(-(-1)^{\bar{\al}+\bar{\al_i}}v_{\al-\al_i})\\&=&-(sgn~\al)m_i^+(\al)v_{\al-\al_i}\\&=&-f_i(v_{\al})
		\end{eqnarray*}
		as $sgn(-\al)=-sgn(\al)$ and $m_i^-(-\al)=m_i^+(\al)$. If $\al\in-R^+$, then $-\al+\al_i\in R^+$ and 
		\begin{eqnarray*}
			\om e_i\om^{-1}(v_{\al})&=&-\om e_i(v_{-\al})\\&=&-(sgn-\al)m_i^-(-\al)\om(v_{-\al+\al_i})\\&=&-(sgn-\al)m_i^-(-\al)(-v_{\al-\al_i})\\&=&-(sgn~\al)m_i^+(\al)v_{\al-\al_i}\\&=&-f_i(v_{\al}).
		\end{eqnarray*}
		If $\al\in R^+$, then $-\al-\al_i\in-R^+$ and
		\begin{eqnarray*}
			\om f_i\om^{-1}(v_{\al})&=&-(-1)^{\bar{\al}}\om f_i(v_{-\al})\\&=&-(-1)^{\bar{\al}}(sgn-\al)m_i^+(-\al)\om(v_{-\al-\al_i})\\&=&-(-1)^{\bar{\al}}(sgn-\al)m_i^+(-\al)(-(-1)^{\bar{\al}+\bar{\al_i}}v_{\al+\al_i})\\&=&-(-1)^{\bar{i}}(sgn~\al)m_i^-(\al)v_{\al-\al_i}\\&=&-(-1)^{\bar{i}}e_i(v_{\al})
		\end{eqnarray*}
		as $sgn(-\al)=-sgn(\al)$ and $m_i^+(-\al)=m_i^-(\al)$. If $\al\in-R^+$, then $-\al-\al_i\in R^+$ and 
		\begin{eqnarray*}
			\om f_i\om^{-1}(v_{\al})&=&-\om f_i(v_{-\al})\\&=&-(sgn-\al)m_i^+(-\al)\om(v_{-\al-\al_i})\\&=&-(sgn-\al)m_i^+(-\al)(-v_{\al+\al_i})\\&=&-(-1)^{\bar{i}}(sgn~\al)m_i^-(\al)v_{\al+\al_i}\\&=&-(-1)^{\bar{i}}e_i(v_{\al}).
		\end{eqnarray*}
		For $\bar{i}=1$ and $\al\in R^{re}$, $e_i(v_{\al})=m_i^-(\al)v_{\al+\al_i}$ and $f_i(v_{\al})=(sgn~\al)m_i^+(\al)v_{\al-\al_i}$.
		If $\al\in R^+$, then $-\al+\al_i\in-R^+$ and
		\begin{eqnarray*}
			\om e_i\om^{-1}(v_{\al})&=&-(-1)^{\bar{\al}}\om e_i(v_{-\al})\\&=&-(-1)^{\bar{\al}}m_i^-(-\al)\om(v_{-\al+\al_i})\\&=&-(-1)^{\bar{\al}}m_i^-(-\al)(-(-1)^{\bar{\al}+\bar{\al_i}}v_{\al-\al_i})\\&=&-m_i^+(\al)v_{\al-\al_i}\\&=&-(sgn~\al)f_i(v_{\al})
		\end{eqnarray*}
		as $sgn(\al)=1$ and $m_i^-(-\al)=m_i^+(\al)$. If $\al\in-R^+$, then $-\al+\al_i\in R^+$ and $sgn(\al)=-1$,
		\begin{eqnarray*}
			\om e_i\om^{-1}(v_{\al})&=&-\om e_i(v_{-\al})\\&=&-m_i^-(-\al)\om(v_{-\al+\al_i})\\&=&-m_i^-(-\al)(-v_{\al-\al_i})\\&=&m_i^+(\al)v_{\al-\al_i}\\&=&-(sgn~\al)f_i(v_{\al}).
		\end{eqnarray*}
		If $\al\in R^+$, then $-\al-\al_i\in-R^+$ and
		\begin{eqnarray*}
			\om f_i\om^{-1}(v_{\al})&=&-(-1)^{\bar{\al}}\om f_i(v_{-\al})\\&=&-(-1)^{\bar{\al}}(sgn-\al)m_i^+(-\al)\om(v_{-\al-\al_i})\\&=&-(-1)^{\bar{\al}}(sgn-\al)m_i^+(-\al)(-(-1)^{\bar{\al}+\bar{\al_i}}v_{\al+\al_i})\\&=&-(-1)^{\bar{i}}m_i^-(\al)v_{\al+\al_i}\\&=&-(-1)^{\bar{i}}e_i(v_{\al})
		\end{eqnarray*}
		as $sgn(-\al)=-1$ and $m_i^+(-\al)=m_i^-(\al)$. If $\al\in-R^+$, then $-\al-\al_i\in R^+$ and $sgn(-\al)=1$,
		\begin{eqnarray*}
			\om f_i\om^{-1}(v_{\al})&=&-\om f_i(v_{-\al})\\&=&-(sgn-\al)m_i^+(-\al)\om(v_{-\al-\al_i})\\&=&-(sgn-\al)m_i^+(-\al)(-v_{\al+\al_i})\\&=&-(-1)^{\bar{i}}m_i^-(\al)v_{\al+\al_i}\\&=&-(-1)^{\bar{i}}e_i(v_{\al}).
		\end{eqnarray*}
		For $\bar{i}=1$ and $\al\in R^{im}$, $e_i(v_{\al})=(sgn~\al)(\al,\al_i^{\vee})v_{\al+\al_i}$ and $f_i(v_{\al})=(\al,\al_i^{\vee})v_{\al-\al_i}$.
		If $\al\in R^+$, then $-\al+\al_i\in-R^+$ and
		\begin{eqnarray*}
			\om e_i\om^{-1}(v_{\al})&=&-(-1)^{\bar{\al}}\om e_i(v_{-\al})\\&=&-(-1)^{\bar{\al}}(sgn-\al)(-\al,\al_i^{\vee})\om(v_{-\al+\al_i})\\&=&-(-1)^{\bar{\al}}(sgn-\al)(-\al,\al_i^{\vee})(-(-1)^{\bar{\al}+\bar{\al_i}}v_{\al-\al_i})\\&=&-(\al,\al_i^{\vee})v_{\al-\al_i}\\&=&-f_i(v_{\al})
		\end{eqnarray*}
		as $sgn(-\al)=-1$. If $\al\in-R^+$, then $-\al+\al_i\in R^+$ and $sgn(-\al)=1$,
		\begin{eqnarray*}
			\om e_i\om^{-1}(v_{\al})&=&-\om e_i(v_{-\al})\\&=&-(sgn-\al)(-\al,\al_i^{\vee})\om(v_{-\al+\al_i})\\&=&-(sgn-\al)(-\al,\al_i^{\vee})(-v_{\al-\al_i})\\&=&-(\al,\al_i^{\vee})v_{\al-\al_i}\\&=&-f_i(v_{\al}).
		\end{eqnarray*}
		If $\al\in R^+$, then $-\al-\al_i\in-R^+$ and
		\begin{eqnarray*}
			\om f_i\om^{-1}(v_{\al})&=&-(-1)^{\bar{\al}}\om f_i(v_{-\al})\\&=&-(-1)^{\bar{\al}}(-\al,\al_i^{\vee})\om(v_{-\al-\al_i})\\&=&-(-1)^{\bar{\al}}(-\al,\al_i^{\vee})(-(-1)^{\bar{\al}+\bar{\al_i}}v_{\al+\al_i})\\&=&-(-1)^{\bar{i}}(sgn~\al)(\al,\al_i^{\vee})v_{\al-\al_i}\\&=&-(-1)^{\bar{i}}e_i(v_{\al})
		\end{eqnarray*}
		as $sgn(\al)=1$. If $\al\in-R^+$, then $-\al-\al_i\in R^+$ and $sgn(\al)=-1$,
		\begin{eqnarray*}
			\om f_i\om^{-1}(v_{\al})&=&-\om f_i(v_{-\al})\\&=&-(-\al,\al_i^{\vee})\om(v_{-\al-\al_i})\\&=&-(-\al,\al_i^{\vee})(-v_{\al+\al_i})\\&=&-(-1)^{\bar{i}}(sgn~\al)(\al,\al_i^{\vee})v_{\al+\al_i}\\&=&-(-1)^{\bar{i}}e_i(v_{\al}).
		\end{eqnarray*}
	\end{proof}
	
	\section{Lie superalgebra generated by $e_i,f_i$}\label{sec4}
	Consider the Lie superalgebra $\gl(M)$ consisting of all $\C$-linear maps with the super commutator $[\phi,\psi]=\phi\psi-(-1)^{\overline{\phi}\overline{\psi}}\psi\phi$ for homogeneous linear maps. Let $\g:=\langle e_i,f_i:i\in I\rangle\subseteq\gl(M)$ be the sub Lie superalgebra generated by the homogeneous linear operators along with the relations we have in Lemma \ref{lem7}. Hence $\g$ is the linear span of all Lie monomials of level $n~(n\geq 1)$ and $dim~\g<\infty$.
	
	Clearly, $h_i\in\g$ and let $\h\subset\g$ be the subspace spanned by $h_i,i\in I$, which is an abelian subalgebra having a basis $h_i$ ($i\in I$). Let $\h^*$ be the dual of $\h$ and for $\al\in R$ define $\dot{\al}\in\h^*$ by $\dot{\al}(h_j):=(\al,\al_j^{\vee})$ for all $j\in I$. Let $\dot{R}=\{\dot{\al}:\al\in R\}$ and note that the map $\al\mapsto\dot{\al}$ is linear in $\al$. Since $\al_i~(i\in I)$ spans $R$, hence $\dot{\al_i}$ spans $\dot{R}$. For $\la\in\h^*$, we define $$\g^{\la}:=\{x\in\g:[h,x]=\la(h)x\text{ for all }h\in\h\}.$$ Since $[h_j,e_i]=(\al_i,\al_j^{\vee})e_i=\dot{\al_i}(h_j)e_i$ for all $j\in I$, hence $e_i\in\g^{\dot{\al_i}}$; similarly $f_i\in\g^{-\dot{\al_i}}$ and $h_i\in\g^0:=\{x\in\g:[h,x]=0\text{ for all }h\in\h\}$ for all $i\in I$. 
	
	Now, let $\mathfrak{n}^+:=\langle e_i:i\in I\rangle_{Lie}\text{ and }\mathfrak{n}^-:=\langle f_i:i\in I\rangle_{Lie}$ be subalgebras of $\g$ freely generated by $e_i\text{ and }f_i~(i\in I)$ respectively. Set $Q^+:=\left\{\sum\limits_{i\in I}n_i\dot{\al_i}\in\h^*:n_i\in\Z_{\geq0}\text{ for all }i\in I\right\}\backslash\{0\}$ and $Q^-=-Q^+$. 
	\begin{lemma}
		We have $\mathfrak{n}^{\pm}\subseteq\sum\limits_{\la\in Q^{\pm}}\g^{\la}$.
	\end{lemma}
	\begin{proof}
		Let $x\in\mathfrak{n}^+$ be a Lie monomial in $\{e_i:i\in I\}$ of level $n$, we show the result by induction on $n$ that $x\in \g^{\la}$ for some $\la\in Q^+$. If $n=1$, then $x=e_i$ for some $i$ and we noted that $e_i\in\g^{\dot{\al_i}}$. If $n\geq 2$ and $x=[y,z]$ a Lie monomial of level $n$, where $y,z$ are Lie monomials of levels lesser than $n$, then by induction $y\in\g^{\mu}\text{ and }z\in\g^{\nu}\text{ where }\mu,\nu\in Q^+$. By graded Jacobi identity, $[h_i,[y,z]]=[[h_i,y],z]+(-1)^{\overline{h_i}\overline{y}}[y,[h_i,z]]=\mu(h_i)[y,z]+\nu(h_i)[y,z]$ for all $i\in I$. Hence $x\in\g^{\la}\text{ where }\la:=\mu+\nu\in Q^+$ and argument for $\mathfrak{n}^-$ is completely analogous.
	\end{proof}
	\begin{rmk}
		Since for any $g\in\g$ we have a linear map on $M$, then $M$ has a $\g$-module structure. If we define $(\al,\al_i^{\vee}):=\al(h_i)$ then we have $h_i(u_j)=0\text{ and }h_i(v_{\al})=\al(h_i)v_{\al}$, hence the homogeneous basis $\{u_i:i\in I\}\cup\{v_{\al}:\al\in R\}$ is a set of $h_i~(i\in I)$ weight vectors and we have a weight space decomposition $M=M^0\oplus\bigoplus\limits_{\al\in R}M^{\al}$ where $M^0=\text{span}_{\C}\{u_i:i\in I\}\text{ and }M^{\al}=\text{span}_{\C}\{v_{\al}:\al\in R\}$. We can observe that the action of $e_k,f_k~(k\in I)$ on this homogeneous basis by, if $v\in M$ is a weight vector with weight $\la\in R\cup\{0\}$, then either $e_k(v)=0~(\text{similarly }f_k(v)=0)$ or $e_k(v)\text{ and }f_k(v)$ are again wight vectors with weights $\la+\al_k\text{ and }\la-\al_k\text{ respectively}$.
		
	\end{rmk}
	
	\begin{lemma}
		We have a direct sum decomposition $\g=\mathfrak{n}^-\oplus\h\oplus\mathfrak{n}^+$.
	\end{lemma}
	\begin{proof}
		Following the arguments as in (\cite{geck2017construction}, Lemma 4.5) and relations from Lemma \ref{lem7}, it is clear that $\g=\mathfrak{n}^-+\h+\mathfrak{n}^+$. Suppose that $g=n^-+h+n^+=0\text{ where }n^{\pm}\in\mathfrak{n}^{\pm}\text{ and }h\in\h$. Then $0=g(u_i)=n^-(u_i)+h(u_i)+n^+(u_i)=n^-(u_i)+n^+(u_i)$. Then $n^-(u_i)\in M^{-\al}$ and $n^+(u_i)\in M^{\be}$ for some $\al,\be\in R^+$. Hence by the decomposition of $M$, we have $n^{\pm}(u_i)=0$ for all $i\in I$ and all $h\in\h$. Now for $\al\in R^+$, $0=g(v_{\al})=n^-(v_{\al})+h(v_{\al})+n^+(v_{\al})\in M^{\al'}\oplus M^{\al}\oplus M^{\al''}\text{ where }ht(\al')\preceq ht(\al)\preceq ht(\al'')$. Again by the decomposition $n^{\pm}(v_{\al})=0\text{ and }h(v_{\al})=0$, similarly we have this for $\al\in R^-$. Thus we have $n^{\pm}=0\text{ and }h=0$.
	\end{proof}
	\begin{lemma}
		The abelian subalgebra $\h$ spanned by even homogeneous  diagonalizable operators $h_i$ is Cartan subalgebra of $\g$, and we have decomposition $\g=\h\oplus\bigoplus\limits_{\al\in R}\g^{\al}$ where $dim~\g^{\al}=1$ for $\al\in R$.
	\end{lemma}
	\begin{proof}
		Since $\h\subseteq\g^0\text{ and }\mathfrak{n}^{\pm}\subseteq\bigoplus\limits_{\la\in Q^{\pm}}\g^{\la}$. By above lemma we deduce that $\h=\g^0\text{ and }\mathfrak{n}^{\pm}=\sum\limits_{\la\in Q^{\pm}}\g^{\la}$, turns that $\h$ is Cartan subalgebra of $\g$. Let $R'$ be the root system of $\g$ with respect to $\h$, then $R'\subseteq Q^{\pm}$. Since $\dot{\al_i}\in R'$ which spans $R'$, we deduce that $\dot{\Pi}:=\{\dot{\al_i}:i\in I\}$ is a set of simple roots for $R'$. Since $R$ is finite locally finite root supersystem and by \cite{yousofzadeh2016extended} we can recover $R$ from $\Pi$. Also $\dot{\Pi}\cong\Pi$, hence the root system  recovered from $\dot{\Pi}$ is $\dot{R}\cong R$ which the root system of $\g$ and so we have the decomposition.
	\end{proof}
	Now we have a root space decomposition of $\g$ with respect to the Cartan subalgebra $\h$ whose root system is isomorphic to the irreducible root system which we start. We define grading of the Lie superalgebra $\g$ by $\g_{\bar{0}}=\h\oplus\bigoplus\limits_{\al\in R_{\bar{0}}}$ and $\g_{\bar{1}}=\bigoplus\limits_{\al\in R_{\bar{1}}}$. Also as a $\g$-module both $\g$ (via adjoint representation) and $M$ has same set of weights $R$.
	\begin{rmk}\label{rmk1}
	From Lemma \ref{lem8} we have an Automorphism $\widetilde{\om}:\g\rightarrow\g$ by $g\mapsto\om g\om^{-1}$, which satisfies $\widetilde{\om}^2\restriction{\g_{\bar{0}}}=id_{\g_{\bar{0}}}$ and $\widetilde{\om}^2\restriction{\g_{\bar{1}}}=-id_{\g_{\bar{1}}}$, hence $\widetilde{\om}^4=id_{\g}$. 	
	\end{rmk}
	\begin{lemma}\label{lem9}
		We show $\g_{\bar{1}}$ is faithful $\g_{\bar{0}}$-module under adjoint action.
	\end{lemma}
	\begin{proof}
		 First fix the index of unique  simple isotropic root by $i_0\in I$, hence $e_{i_o},f_{i_0}\in\g_{\bar{1}}$. We prove this lemma by contrapositive method, for this we have to show for each non-zero $x\in\g_{\bar{0}}$ there exist non-zero $y\in\g_{\bar{1}}$ such that $[x,y]\neq0$. First we show this for generators of $\g_{\bar{0}}$. For $x=h_{i_0}\in\g_{\bar{0}}$, take $y=[e_{i_0},e_i]\in\g_{\bar{1}}$ with $i\in\{i_0-1,i_0+1\}$. Then $$[h_{i_0},[e_{i_0},e_i]]=[[h_{i_0},e_{i_0}],e_i]+[e_{i_0},[h_{i_0},e_i]]=(\al_i,\al_{i_0})[e_{i_0},e_i]$$ and $$[e_{i_0},e_i](u_{i_0})=|(\al_i,\al_{i_0})|e_{i_0}(v_{\al_i})=|(\al_i,\al_{i_0})|m_{i_0}^-(\al_i)v_{\al_i+\al_{i_0}}\neq0$$ as $(\al_i,\al_{i_0})\neq0\text{ and }\al_i+\al_{i_0}\in R$.
		 
		 For $x=h_i$ with $i\neq i_0\text{ and }i\in\{i_0-1,i_0+1\}$, take $y=e_{i_0}$ and we have $[h_i,e_{i_0}]=(\al_{i_0},\al_i)e_{i_0}$. Since $(\al_{i_0},\al_i)\neq0$ and $e_{i_0}(u_i)=|(\al_{i_0},\al_i^{\vee})|v_{\al_{i_0}}\neq0$, hence $[h_i,e_{i_0}]\neq0$. If $i\notin\{i_0-1,i_0+1\}$, we find $y$ sequentially. Let $j_i\in I\backslash\{i_0\}$ such that $(\al_{i_0},\al_{j_1})\neq0$ and if $(\al_i,\al_{j_1})\neq0$, then take $y=[e_{j_1},e_{i_0}]\in\g_{\bar{1}}$. Since $y(u_{i_0})\neq0$ and so $[h_i,[e_{j_1},e_{i_0}]]=(\al_{j_1},\al_i^{\vee})[e_{j_1},e_{i_0}]\neq0$ as $(\al_{i_0},\al_i)=0$. Suppose $(\al_i,\al_{j_1})=0$, let $j_2\in I\backslash\{i_0,j_1\}$ such that $(\al_{j_1},\al_{j_2})\neq0$ and if $(\al_i,\al_{j_2})\neq0$, then take $y=[e_{j_2},[e_{j_1},e_{i_0}]]\in\g_{\bar{1}}$. Again $y(u_{i_0})\neq0$ and so $[x,y]\neq0$ as $x(u_{i_0})=0$. If not, having choose $j_3,j_4,\ldots,j_k$ as above, since $I$ is finite and let $j_{k+1}\in I\backslash\{i_0,j_1,\ldots,\j_k\}$ such that $(\al_{j_k},\al_{j_{k+1}})\neq0$. Then $(\al_i,\al_{j_{k+1}})\neq0$ and take $y=[e_{j_{k+1}},\cdots[e_{j_2},[e_{j_1},e_{i_0}]]\cdots]\in\g_{\bar{1}}$, clearly $y(u_{i_0})\neq0$ and $[x,y]\neq0$ as mentioned above.
		 
		 Now for $x=e_i\in\g_{\bar{0}}~(i\neq i_0)$ with $i\in\{i_0-1,i_0+1\}$, take $y=e_{i_0}\in\g_{\bar{1}}$ then $[e_i,e_{i_0}](u_{i_0})\neq0$. If $i\notin\{i_0-1,i_0+1\}$, we can find non-zero $y=[e_{j_{k}},\cdots[e_{j_2},[e_{j_1},e_{i_0}]]\cdots]\in\g_{\bar{1}}$ sequentially as above so that $[x,y]\neq0$. Since, from (Rmk. \ref{rmk1}) we have an automorphism on $\g$ such that $\widetilde{\om}(e_i)=-f_i,\widetilde{\om}(f_i)=-(-1)^{\bar{i}}e_i\text{ and }\widetilde{\om}(h_i)=-h_i$, hence by a completely analogous argument, we can show the result for $f_i\in\g_{\bar{0}}~(i\neq i_0)$.
		 
		 Now by induction we show the result for any non-zero Lie monomial of level $n>1$ in $\g_{\bar{0}}$. Let $x\in\g_{\bar{0}}$ be a non-zero Lie monomial of level $n>1$, then $x\in \g^{\al}$ for some $\al\in R_{\bar{0}}$. Since from (\cite{gorelik2017generalized}, Lemma 2.2) there exist an $\be\in R^{im}$ such that $(\al,\be)\neq 0$  and so $\al\pm\be\in R$. Hence there is a non-zero $y\in\g^{\be}\subseteq\g_{\bar{1}}$ such that $[x,y]\in\g^{\al\pm\be}$ is non-zero and we have the result.
	\end{proof}
	\begin{lemma}\label{lem10}
		We show the $\g_{\bar{0}}$-module $\g_{\bar{1}}$ is either irreducible or any non-zero submodule $\mathfrak{a}$ of $\g_{\bar{1}}$ be such that $[\g_{\bar{1}},\mathfrak{a}]=\g_{\bar{0}}$.
	\end{lemma}
	\begin{rmk}
		By above lemma, for any $\g_{\bar{0}}$-submodule $\mathfrak{a}$ of $\g_{\bar{1}}$ we have $\g_{\bar{0}}=[\g_{\bar{1}},\mathfrak{a}]\subseteq[\g_{\bar{1}},\g_{\bar{1}}]\subseteq\g_{\bar{0}}$, hence $[\g_{\bar{1}},\g_{\bar{1}}]=\g_{\bar{0}}$.
	\end{rmk}
	\begin{lemma}\label{lem11}
		We show that $[\g_{\bar{0}},\g_{\bar{1}}]=\g_{\bar{1}}$.
	\end{lemma}
\begin{proof}
	Since $[\g_{\bar{0}},\g_{\bar{1}}]\subseteq\g_{\bar{1}}$ and $\g_{\bar{1}}=\bigoplus\limits_{\al\in R_{\bar{1}}}\g^{\al}$. Also $[\h,\g^{\al}]=\g^{\al}$, hence $\g_{\bar{1}}\subseteq[\g_{\bar{0}},\g_{\bar{1}}]$.
\end{proof}
\begin{lemma}
	The Lie superalgebra $\g$ is simple.
\end{lemma}
\begin{proof}
	Since from (\cite{musson2012lie}, Lemma 1.2.1) and from (\cite{scheunert2006theory}, Chap. II,$\S2.1$), we have the sufficient condition for simplicity of $\g$. Also in Lemmas (\ref{lem9}, \ref{lem10}, \ref{lem11}) we establish the requirement.
\end{proof}
    \section{Quiver diagram of Irreducible GRS}\label{sec5}
In this section we list all irreducible GRS with the set of positive even and odd roots, a distinguished basis with its Dynkin diagram and their quiver diagrams so that the roots are arranged in order to satisfy the Lemma \ref{lem6}.

1. $D(2,1;a)$
$$R_{\bar{0}}^+=\{2\ep_1,2\ep_2,2\ep_3\}$$ $$R_{\bar{1}}^+=\{\ep_2\pm\ep_1\pm\ep_3\}$$ $$\Pi=\{2\ep_1,\ep_2-\ep_1-\ep_3,2\ep_3\}$$
\begin{center}

			
		}	
		
	\end{center}

	{\bf Acknowledgement:} We are grateful to Malihe Yousofzadeh for fruitful discussions. The second author's work is supported by SERB (ANRF) grant (File No. SRG/2023/002255).

	\small
	\addcontentsline{toc}{section}{\bf References}
	\nocite{*}

	\bibliographystyle{plain}
	\bibliography{ref.bib}

\end{document}